\documentclass[12pt,oneside]{amsart}

\usepackage{amssymb}
\usepackage{amscd}

\title[Examples of plane rational curves with two Galois points]{Examples of plane rational curves with two Galois points in positive characteristic} 
\author{Satoru Fukasawa and Katsushi Waki}

\address{Department of Mathematical Sciences, Faculty of Science, Yamagata University, Kojirakawa-machi 1-4-12, Yamagata 990-8560, Japan} 
\email{s.fukasawa@sci.kj.yamagata-u.ac.jp} 
\email{waki@sci.kj.yamagata-u.ac.jp}

\thanks{The first author was partially supported by JSPS KAKENHI Grant Number 16K05088.}

\newtheorem{theorem}{Theorem}

\newtheorem{problem}{Problem}

\newtheorem{lemma}{Lemma}

\begin{document}
\begin{abstract}  
We present four new examples of plane rational curves with two Galois points in positive characteristic, and determine the number of Galois points for three of them. 
Our results are related to a problem on projective linear groups. 
\end{abstract}

\maketitle 

\section{Introduction} 
Let $C \subset \Bbb P^2$ be an irreducible plane curve over an algebraically closed field $k$ of characteristic $p \ge 0$ with $k(C)$ as its function field.  
For a point $P \in \Bbb P^2$, if the function field extension $k(C)/\pi_P^*k(\Bbb P^1)$ induced by the projection $\pi_P$ is Galois, then $P$ is called a Galois point for $C$. 
This notion was introduced by Yoshihara (\cite{miura-yoshihara, yoshihara1}).
Furthermore, if a Galois point $P$ is a smooth point of $C$ (resp. a point in $\Bbb P^2 \setminus C$), then $P$ is said to be inner (resp. outer). 
The associated Galois group at $P$ is denoted by $G_P$. 
Determining the number of Galois points in general is difficult. 
For example, there are not so many examples of plane curves with two Galois points (see the Tables in \cite{yoshihara-fukasawa}). 
In this note, we present four examples of plane rational curves with two Galois points, which update the Tables in \cite{yoshihara-fukasawa}. 

Hereafter, we assume that $p \ge 3$ and $q \ge 5$ is a power of $p$. 

\begin{theorem} \label{main1}
Let $C_1$ be the plane curve of degree $q$ which is the image of the morphism 
$$ \varphi_1: \Bbb P^1 \rightarrow \Bbb P^2; \ (s:t) \mapsto (s^{\frac{q+1}{2}}t^{\frac{q-1}{2}}:(s-t)t^{q-1}:s^q-st^{q-1}). $$
Then: 
\begin{itemize}
\item[(a)] The point $P_1:=(0:1:0)=\varphi_1(0:1)$ is an inner Galois point with $G_{P_1} \cong D_{q-1}$, where $D_{q-1}$ is the dihedral group of order $q-1$. 
\item[(b)] The point $P_2:=(1:0:0)=\varphi_1(1:1)$ is an inner Galois point with $G_{P_2} \cong \Bbb Z/(q-1)\Bbb Z$.  
\item[(c)] The number of inner Galois points on $C_1$ is exactly two. 
\end{itemize}
\end{theorem}

\begin{theorem} \label{main2}
Let $C_2$ be the plane curve of degree $q$ which is the image of the morphism 
$$ \varphi_2: \Bbb P^1 \rightarrow \Bbb P^2; \ (s:t) \mapsto (s^{\frac{q+1}{2}}t^{\frac{q-1}{2}}:(s-t)^{\frac{q+1}{2}}t^{\frac{q-1}{2}}:s^q-st^{q-1}). $$
Then: 
\begin{itemize}
\item[(a)] The point $P_1:=(0:1:0)=\varphi_2(0:1)$ is an inner Galois point with $G_{P_1} \cong D_{q-1}$.  
\item[(b)] The point $P_2:=(1:0:0)=\varphi_2(1:1)$ is an inner Galois point with $G_{P_2} \cong D_{q-1}$.  
\item[(c)] The number of inner Galois points on $C_2$ is exactly two. 
\end{itemize}
\end{theorem}

For the case where the characteristic is zero and $C$ is rational, Yoshihara \cite[Lemma 13]{yoshihara2} asserts that cyclic and dihedral groups do not appear as Galois groups of outer Galois points at the same time.   
To the contrary, in positive characteristic, we present the following example. 

\begin{theorem} \label{main3} 
Let $C_3$ be the plane curve of degree $q+1$ which is the image of the morphism 
$$ \varphi_3: \Bbb P^1 \rightarrow \Bbb P^2; \ (s:t) \mapsto (s^{\frac{q+1}{2}}t^{\frac{q+1}{2}}: (s+t)^{q+1}: s^{q+1}+\gamma t^{q+1}), $$
where $\gamma \in \Bbb F_{q} \setminus \{0, \pm 1\}$. 
Then: 
\begin{itemize}
\item[(a)] The point $P_1:=(0:1:0)$ is an outer Galois point with $G_{P_1} \cong D_{q+1}$.  
\item[(b)] The point $P_2:=(1:0:0)$ is an outer Galois point with $G_{P_2} \cong \Bbb Z/(q+1)\Bbb Z$.  
\item[(c)] The number of outer Galois points for $C_3$ is exactly two. 
\end{itemize}
\end{theorem}

\begin{theorem} \label{main4} 
Let $C_4$ be the plane curve of degree $q+1 >6$ which is the image of the morphism 
$$ \varphi_4: \Bbb P^1 \rightarrow \Bbb P^2; \ (s:t) \mapsto (s^{\frac{q+1}{2}}t^{\frac{q+1}{2}}: (s+t)^{\frac{q+1}{2}}(s+\gamma t)^{\frac{q+1}{2}}: s^{q+1}-\gamma t^{q+1}), $$
where $\gamma \in \Bbb F_{q} \setminus \{\pm 1\}$ and $\gamma^{\frac{q-1}{2}}=1$. 
Then: 
\begin{itemize}
\item[(a)] The point $P_1:=(0:1:0)$ is an outer Galois point with $G_{P_1} \cong D_{q+1}$.  
\item[(b)] The point $P_2:=(1:0:0)$ is an outer Galois point with $G_{P_2} \cong D_{q+1}$.  
\end{itemize}
\end{theorem}

Our resluts are closely related to the following problem on projective linear groups (cf. \cite[Theorem 1]{fukasawa2}). 

\begin{problem}
Let $k$ be a field and $X=\Bbb P^1(k)$ the projective line over $k$. 
We consider the following two conditions for a pair $(H_1, H_2)$ of different finite subgroups $H_1$ and $H_2 \subset {\rm PGL}(2, k)$: 
\begin{itemize}
\item[(a)] $H_1 \cap H_2=\{1\}$. 
\item[(b)] There exist different points $P_1$ and $P_2 \in X$ such that 
$$\{\sigma(P_2) \ | \ \sigma \in H_1\setminus \{1\}\}=\{\tau(P_1) \ | \ \tau \in H_2 \setminus \{1\}\}$$  
(with multiplicities).  
\end{itemize}
When does a pair $(H_1, H_2)$ with (a) and (b) exist?  
\end{problem}

\section{Proof of Theorems \ref{main1} and \ref{main2}} 

We assume that $\alpha \in \Bbb F_q$ is a primitive element. 
The following two lemmas are easily proved. 

\begin{lemma} \label{dihedral} 
Let $\sigma$ and $\tau \in {\rm Aut}(\Bbb P^1) \cong {\rm PGL}(2, k)$ be represented by the matrices 
$$ A_{\sigma}=\left(\begin{array}{cc}
1 & 0 \\
0 & \alpha^2 
\end{array}\right) \ \mbox{ and } \
A_{\tau}=\left(\begin{array}{cc}
0 & 1 \\
\alpha & 0 
\end{array}\right)$$
respectively, that is, $\sigma(s, t)=(s, t)A_{\sigma}$ and $\tau(s,t)=(s,t)A_{\tau}$.  
Then: 
\begin{itemize} \label{cyclic} 
\item[(a)] The group $H$ generated by $\sigma$ and $\tau$ is isomorphic to $D_{q-1}$.  
\item[(b)] The rational function $f(s)=s^{\frac{q-1}{2}}-s^{-\frac{q-1}{2}} \in k(\Bbb P^1)=k(s)$ is invariant under the action by $H$. 
\end{itemize}
\end{lemma}

\begin{lemma} \label{cyclic}
Let $\eta \in {\rm Aut}(\Bbb P^1)$ be represented by the matrix 
$$ A_{\eta}=\left(\begin{array}{cc}
1 & 0 \\
\alpha-1 & \alpha 
\end{array}\right). $$  
Then: 
\begin{itemize}
\item[(a)] The order of $\eta$ is $q-1$.  
\item[(b)] The rational function $g(s)=\frac{s-1}{s^q-s} \in k(\Bbb P^1)=k(s)$ is invariant under the action by $\eta^*$. 
\end{itemize}
\end{lemma}

When $R_1$ and $R_2$ are different points in $\Bbb P^2$, the line passing through $R_1$ and $R_2$ is denoted by $\overline{R_1R_2}$. 

\begin{proof}[Proof of Theorem \ref{main1}]
The composite (rational) map $\pi_{P_1} \circ \varphi_1$ is given by 
$$ (s:t) \mapsto (s^{\frac{q+1}{2}}t^{\frac{q-1}{2}}:s^q-st^{q-1})=(s^{\frac{q-1}{2}}t^{\frac{q-1}{2}}:s^{q-1}-t^{q-1}), $$
since $\pi_{P_1}$ is represented by $(X:Y:Z) \mapsto (X:Z)$. 
By this expression, the degree of $\pi_{P_1}\circ \varphi_1$ is $q-1$ and hence, $\varphi_1$ is birational onto the image $C_1$. 
Note that 
$$ \pi_{P_1}\circ\varphi_1(s:1)=\left(s^{\frac{q-1}{2}}:s^{q-1}-1\right)=\left(1:f(s)\right), $$
where $f(s)$ is the rational function as in Lemma \ref{dihedral}. 
By Lemma \ref{dihedral},  the rational function $f(s)$ is invariant under the action by $H \cong D_{q-1}$. 
Therefore, $k(s)/k(f(s))$ is a Galois extension and hence, $P_1$ is a Galois point. 
In this case, $G_{P_1} \cong D_{q-1}$. 
Assertion (a) in Theorem \ref{main1} follows. 
Further, 
$$ \pi_{P_2}\circ\varphi_1(s:1)=\left(s-1: s^q-s\right)=(g(s):1), $$
where $g(s)$ is the rational function as in Lemma \ref{cyclic}. 
By Lemma \ref{cyclic},  the rational function $g(s)$ is invariant under the action by $ \eta^*$. 
Therefore, $k(s)/k(g(s))$ is a Galois extension and hence, $P_2$ is a Galois point. 
In this case, $G_{P_2} \cong \Bbb Z/(q-1)\Bbb Z$. 
Assertion (b) in Theorem \ref{main1} follows. 

We prove that the set of all inner Galois points for $C_1$ is equal to $\{P_1, P_2\}$. 
Let $Q:=\varphi_1(1:0)=(0:0:1)$. 
Then, the composite map $\pi_Q\circ\varphi_1$ is given by 
$$ (s:t) \mapsto (s^{\frac{q+1}{2}}t^{\frac{q-1}{2}}:(s-t)t^{q-1})=(s^{\frac{q+1}{2}}:(s-t)t^{\frac{q-1}{2}}).  $$ 
Since the degree of $\pi_Q \circ \varphi_1$ is $(q+1)/2$, the point $Q$ is a singular point of $C_1$ with multiplicity $(q-1)/2$. 
The ramification indices of $\pi_Q \circ \varphi_1$ at $(0:1)$ and at $(1:0)$ are equal to $(q+1)/2$ and $(q-1)/2$ respectively. 
For the function $h(t)=(1-t)t^{\frac{q-1}{2}}$, 
$$h'(t)=-\frac{1}{2}t^{\frac{q-3}{2}}(1+t).  $$ 
Therefore, three points $(0:1)$, $(1:0)$ and $(1:-1)$ are all ramification points for $\pi_Q \circ \varphi_1$. 
Furthermore, the ramification index at $(1:-1)$ is $2$.
Note that $\varphi_1^{-1}(Q)$ consists of a unique point $(1:0)$. 
Therefore, for each point $R \in C_1 \setminus \{Q\}$, the map $\pi_R \circ \varphi_1$ is ramified at $(1:0)$ with index $\ge (q-1)/2$. 
Assume that $R$ is an inner Galois point. 
It follows from \cite[III. 7.2]{stichtenoth} that $\pi_R \circ \varphi_1$ is ramified at $(1:0)$ with index $(q-1)/2$ or $q-1$. 
If the index at $(1:0)$ is $q-1$, then $R=P_2$. 
Assume that the index at $(1:0)$ is $(q-1)/2$. 
Then, there exists a ramification point $\hat{S} \in \Bbb P^1$ with index $(q-1)/2$ such that $\varphi_1(\hat{S})\ne Q$ and $\varphi_1(\hat{S}) \in \overline{RQ}$. 
Considering $\pi_Q \circ \varphi_1$, $\hat{S}=(0:1)$ or $(1:-1)$. 
If $\hat{S}=(0:1)$, then $R=\varphi_1(0:1)=P_1$, since $C_1 \cap \overline{P_1Q}=\{P_1, Q\}$. 
Assume that $\hat{S}=(1:-1)$. 
Then, $(q-1)/2=2$, and hence, $q=5$ and $R \in \overline{Q\varphi_1(1:-1)}$. 
However, this is a contradiction, because the point $\varphi_1(1:-1)=(1:2:0)=\varphi_1(2:1)$ is a singular point and $C \cap \overline{Q \varphi_1(1:-1)}=\{Q, \varphi_1(1:-1)\}$. 
We complete the proof of Theorem \ref{main1}(c). 
\end{proof} 

\begin{proof}[Proof of Theorem \ref{main2}]
Assertion (a) in Theorem \ref{main2} is similar to assertion (a) in Theorem \ref{main1}. 
The composite map $\pi_{P_2} \circ \varphi_2$ is given by  
$$ (s:1) \mapsto ((s-1)^{\frac{q+1}{2}}:s^q-s)=\left(1:f(s-1)\right). $$
Then, the induced field extension is $k(u)/k(f(u))$, where $u=s-1$. 
This is a Galois extension, similar to assertion (a) in Theorem \ref{main1}. 
Assertion (b) in Theorem \ref{main2} follows. 

We prove that the set of all inner Galois points for $C_2$ is equal to $\{P_1, P_2\}$. 
Let $Q:=\varphi_2(1:0)=(0:0:1)$. 
Then, the composite map $\pi_Q\circ\varphi_2$ is given by 
$$ (s:t) \mapsto (s^{\frac{q+1}{2}}t^{\frac{q-1}{2}}:(s-t)^{\frac{q+1}{2}}t^{\frac{q-1}{2}})=(s^{\frac{q+1}{2}}:(s-t)^{\frac{q+1}{2}}).  $$ 
Since the degree of $\pi_Q \circ \varphi_2$ is $(q+1)/2$, the point $Q$ is a singular point of $C_2$ with multiplicity $(q-1)/2$.
Further, by this expression, $\pi_Q \circ \varphi_2$ is a cyclic covering and all ramification points are $(0:1)$ and $(1:1)$ with index $(q+1)/2$.  
Note that $\varphi_2^{-1}(Q)$ consists of a unique point $(1:0)$, and the intersection multiplicity of $C_2 \cap L$ at $Q$ is at most $\frac{q-1}{2}+1$ for any line $L$ passing through $Q$. 
Therefore, for each point $R \in C_2 \setminus \{Q\}$, the map $\pi_R \circ \varphi_2$ is ramified at $(1:0)$ with index $(q-1)/2$ or $(q+1)/2$. 
Assume that $R$ is an inner Galois point. 
It follows from \cite[III. 7.2]{stichtenoth} that $\pi_R \circ \varphi_2$ is ramified at $(1:0)$ with index $(q-1)/2$. 
Then, there exists a ramification point $\hat{S} \in \Bbb P^1$ with index $(q-1)/2$ such that $\varphi_2(\hat{S})\ne Q$ and $\varphi_2(\hat{S}) \in \overline{RQ}$. 
Considering $\pi_Q\circ\varphi_2$, $\hat{S}=(0:1)$ or $(1:1)$. 
Then, $R=P_1$ or $P_2$, since $C_2 \cap \overline{QP_i}=\{Q, P_i\}$ for $i=1, 2$. 
We complete the proof of Theorem \ref{main2}(c). 
\end{proof} 

\section{Proof of Theorems \ref{main3} and \ref{main4}}
\begin{proof}[Proof of Theorem \ref{main3}]
Let $Q:=\varphi_3(1:0)=(0:1:1)$. 
Then, the composite map $\pi_{Q} \circ \varphi_3$ is given by 
$$ (s:t) \mapsto (s^{\frac{q+1}{2}}t^{\frac{q+1}{2}}:(s+t)^{q+1}-(s^{q+1}+\gamma t^{q+1}))=(s^{\frac{q+1}{2}}t^{\frac{q-1}{2}}:s^{q}+st^{q-1}+(1-\gamma)t^q). $$
Since $\gamma \ne 1$, the degree of $\pi_Q \circ \varphi_3$ is $q$ and hence, $\varphi_3$ is birational onto the image $C_3$. 

We consider the point $P_1$. 
Since $\gamma \ne 0$, $P_1 \in \Bbb P^2 \setminus C_3$. 
The composite map $\pi_{P_1} \circ \varphi_3$ is given by 
$$ (s:1) \mapsto (s^{\frac{q+1}{2}}:s^{q+1}+\gamma)=(1:s^{\frac{q+1}{2}}+\gamma s^{-\frac{q+1}{2}}).  $$ 
The rational function $s^{\frac{q+1}{2}}+\gamma s^{-\frac{q+1}{2}}$ is invariant under the actions $s \mapsto \delta s^{-1}$ and $s \mapsto \zeta s$, where $\delta$ is a $(q+1)/2$-th root of $\gamma$ and $\zeta$ is a $(q+1)/2$-th root of unity. 
Therefore, the extension $k(\Bbb P^1)/\pi_{P_1}^*k(\Bbb P^1)$ is a Galois extension and hence, $P_1$ is a Galois point. 
In this case, $G_{P_1} \cong D_{q+1}$. 
Assertion (a) in Theorem \ref{main3} follows. 
We consider the point $P_2$. 
Since $\gamma \ne -1$, $P_2 \in \Bbb P^2 \setminus C_3$.
The composite map $\pi_{P_2}\circ\varphi_3$ is given by
$$ (s:1) \mapsto \left((s+1)^{q+1}:s^{q+1}+\gamma \right).  $$ 
Note that 
$$ (1+\gamma)(s^{q+1}+\gamma)-\gamma(s+1)^{q+1}=(s-\gamma)^{q+1}. $$
Therefore, the extension $k(\Bbb P^1)/\pi_{P_2}^*k(\Bbb P^1)$ is a Galois extension and hence, $P_2$ is a Galois point. 
In this case, $G_{P_2} \cong \Bbb Z/(q+1)\Bbb Z$. 
Assertion (b) in Theorem \ref{main3} follows. 

We prove that the set of all outer Galois points for $C_3$ is equal to $\{P_1, P_2\}$. 
Note that the rank of the matrix
$$ \left(\begin{array}{c} 
\varphi_3 \\
\frac{d\varphi_3}{ds} \\
\end{array} \right)=
\left(\begin{array}{ccc}
s^{\frac{q+1}{2}} & (s+1)^{q+1} & s^{q+1}+\gamma \\
\frac{q+1}{2}s^{\frac{q-1}{2}} & (s+1)^q & s^q 
\end{array} \right)
$$
is two for each $s$, that is, the differential map of $\varphi_3$ is injective at each point $(s:1) \in \Bbb P^1$. 
We consider the Hessian matrix $H$ of $\varphi_3$: 
$$ \left(\begin{array}{c} 
\varphi_3 \\
\frac{d\varphi_3}{ds} \\
\frac{d^2\varphi_3}{ds^2}
\end{array} \right)=
\left(\begin{array}{ccc}
s^{\frac{q+1}{2}} & (s+1)^{q+1} & s^{q+1}+\gamma \\
\frac{q+1}{2}s^{\frac{q-1}{2}} & (s+1)^q & s^q \\
\frac{q^2-1}{4}s^{\frac{q-3}{2}} & 0 & 0 
\end{array} \right). 
$$
Note that $\det H=0$ if and only if $s=0, -1, \gamma$. 
It follows that all flexes of $C_3$ are $(1:0), (0:1), (-1:1)$ and $(\gamma:1)$ (see \cite[Section 7.6]{hkt}). 
If $R$ is an outer Galois point such that the order of $G_R$ is at least five, then there exists a ramification point with index at least three (see, for example, \cite[Theorem 11.91]{hkt}). 
Then, $R$ is contained in the tangent line at a flex. 
If $R \ne P_2$, then $R$ is contained in the line defined by $X=0$, which passes through points $\varphi_3(1:0)$, $\varphi_3(0:1)$ and $P_1$. 
It follows from \cite[III. 7.2, 8.2]{stichtenoth} that there exist subgroups $H_1 \subset G_{P_1}$ and $H_2 \subset G_R$ of order $\frac{q+1}{2}$ fixing the point $(1:0)$. 
Then, $H_1$ and $H_2$ are normal subgroups of $G_{P_1}$ and $G_R$ respectively, and hence, they fix points $(1:0)$ and $(0:1)$. 
By a property of ${\rm PGL}(2, k)$, it follows that $H_1=H_2$, that is, $G_{P_1}\cap G_{R} \ne \{1\}$.
This is a contradiction (see, for example, \cite[Lemma 7]{fukasawa1}).   
We complete the proof of Theorem \ref{main3}.  
\end{proof}

\begin{proof}[Proof of Theorem \ref{main4}]
Let $Q:=\varphi_4(1:0)=(0:1:1)$. 
Then, the composite map $\pi_{Q} \circ \varphi_4$ is given by 
$$ (s:t) \mapsto (s^{\frac{q+1}{2}}t^{\frac{q+1}{2}}:(s+t)^{\frac{q+1}{2}}(s+\gamma t)^{\frac{q+1}{2}}-(s^{q+1}-\gamma t^{q+1})). $$ 
Note that the coefficient of $t^{q+1}$ and $s^qt$ for the function $(s+t)^{\frac{q+1}{2}}(s+\gamma t)^{\frac{q+1}{2}}-(s^{q+1}-\gamma t^{q+1})$ is $\gamma^{\frac{q+1}{2}}+\gamma=2\gamma \ne 0$ and $\frac{q+1}{2}(\gamma+1) \ne 0$ respectively. 
The degree of $\pi_Q \circ \varphi_3$ is $q$ and hence, $\varphi_4$ is birational onto the image $C_4$. 

We consider the point $P_1$. 
Since $\gamma \ne 0$, $P_1 \in \Bbb P^2 \setminus C_4$. 
The composite map $\pi_{P_1} \circ \varphi_4$ is given by 
$$ (s:1) \mapsto (s^{\frac{q+1}{2}}:s^{q+1}-\gamma)=(1:s^{\frac{q+1}{2}}-\gamma s^{-\frac{q+1}{2}}).  $$ 
The rational function $s^{\frac{q+1}{2}}-\gamma s^{-\frac{q+1}{2}}$ is invariant under the actions $s \mapsto \delta s^{-1}$ and $s \mapsto \zeta s$, where $\delta$ is a $(q+1)/2$-th root of $-\gamma$ and $\zeta$ is a $(q+1)/2$-th root of unity. 
Therefore, the extension $k(\Bbb P^1)/\pi_{P_1}^*k(\Bbb P^1)$ is a Galois extension and hence, $P_1$ is a Galois point. 
In this case, $G_{P_1} \cong D_{q+1}$. 
Assertion (a) in Theorem \ref{main4} follows. 
We consider the point $P_2$. 
Since $(-1)^{q+1}-\gamma \ne 0$ and $(-\gamma)^{q+1}-\gamma=\gamma(\gamma-1)^q\ne 0$, it follows that $P_2 \in \Bbb P^2 \setminus C_4$. 
The composite map $\pi_{P_2} \circ \varphi_4$ is given by 
$$ (s:1) \mapsto ((s+1)^{\frac{q+1}{2}}(s+\gamma)^{\frac{q+1}{2}}:s^{q+1}-\gamma). $$
Note that 
$$ \frac{1}{1-\gamma}\{-\gamma (s+1)^{q+1}+(s+\gamma)^{q+1}\}=s^{q+1}-\gamma. $$
Let 
$$ u:=\frac{s+\gamma}{s+1} \ \mbox{ and } \ h(u):=-\gamma u^{-\frac{q+1}{2}}+u^{\frac{q+1}{2}}. $$
Then, $k(\Bbb P^1)/\pi_{P_2}^*k(\Bbb P^1)=k(u)/k(h(u))$. 
This is a Galois extension, similar to assertion (a) in Theorem \ref{main4}. 
In this case, $G_{P_2} \cong D_{q+1}$. 
Assertion (b) in Theorem \ref{main4} follows. 
\end{proof}


\begin{thebibliography}{20} 
\bibitem{fukasawa1} S. Fukasawa, Classification of plane curves with infinitely many Galois points, J. Math. Soc. Japan {\bf 63} (2011), 195--209. 
\bibitem{fukasawa2} S. Fukasawa, A birational embedding of an algebraic curve into a projective plane with two Galois points, preprint, arXiv:1611.03953. 
\bibitem{hkt} J. W. P. Hirschfeld, G. Korchm\'{a}ros and F. Torres, {\it Algebraic curves over a finite field}, Princeton Univ. Press, Princeton, 2008. 
\bibitem{miura-yoshihara} K. Miura and H. Yoshihara, Field theory for function fields of plane quartic curves, J. Algebra {\bf 226} (2000), 283--294. 
\bibitem{stichtenoth} H. Stichtenoth, {\it Algebraic function fields and codes}, Universitext, Springer-Verlag, Berlin, 1993.
\bibitem{yoshihara1} H. Yoshihara, Function field theory of plane curves by dual curves, J. Algebra {\bf 239} (2001), 340--355. 
\bibitem{yoshihara2} H. Yoshihara, Galois points for plane rational curves, Far East J. Math. {\bf 25} (2007), 273--284;  
Errata, ibid. {\bf 29} (2008), 209--212.
\bibitem{yoshihara-fukasawa} H. Yoshihara and S. Fukasawa, List of problems, available at: \\ http://hyoshihara.web.fc2.com/openquestion.html
\end{thebibliography}
\end{document}